\documentclass[12pt]{article}

\usepackage{amssymb}
\usepackage{amsthm}
\usepackage{amsmath}
\usepackage[numbers]{natbib}

\usepackage{geometry}
\newtheorem{lemma}{Lemma}[section]
\newtheorem{theorem}{Theorem}[section]
\newtheorem{korollar}{Corollary}[section]
\newtheorem{anmerk}{Remark}[section]

\newcommand*{\N}{\mathbb{N}}
\newcommand*{\R}{\mathbb{R}} 

\begin{document}

\title{On Hahn polynomial expansion of a continuous function of bounded variation}
\author{Ren\'e Goertz and Philipp \"Offner}
\date{\today}
\maketitle

\begin{abstract}
We consider the well-known method of least squares on an equidistant grid with $N+1$ nodes on the interval $[-1,1]$. We investigate the following problem: 
For which ratio $N/n$ and which functions, do we have pointwise convergence of the least square operator ${LS}_n^N:\mathcal{C}\left[-1,1\right]\rightarrow\mathcal{P}_n$? 
To solve this problem we investigate the relation between the Jacobi polynomials $P_k^{\alpha,\beta}$ and the Hahn polynomials $Q_k\left(\cdot;\alpha,\beta,N\right)$. Thereby we describe the least square operator ${LS}_n^N$ by the expansion of a function by Hahn polynomials. In particular we present the following result:
The series expansion $\sum_{k=0}^n{\hat{f} Q_k}$ of a function $f$ by Hahn polynomials $Q_k$ converges pointwise, if the series expansion $\sum_{k=0}^n{\hat{f} P_k}$ of the function $f$ by Jacobi polynomials $P_k$ converges pointwise and if ${n^4}/N\rightarrow 0$ for $n,N\rightarrow\infty$. Furthermore we obtain the following result:
Let $f\in\left\{g\in\mathcal{C}^1\left[-1,1\right]:g^\prime\in\mathcal{BV}\left[-1,1\right]\right\}$ and let $(N_n)_{n}$ be a sequence of natural numbers with ${n^4}/{N_n}\rightarrow 0$. Then the least square method ${LS}_n^{N_n}[f]$ converges for each $x\in[-1,1]$.
\end{abstract}

\section{Introduction and statement of the main results}
\label{introduction}

It is over $200$ years ago since Legendre, Gau\ss\, and others started working with the method of least squares (cf., e.g., \cite{merriman1877history}). 
Since then, the method is used in many areas of mathematics and is nowaday a basic tool for every applied mathematician.
Our focus in this paper is the pure approximation property of the method on a grid with $N+1$ nodes on the interval $[-1,1]$.
\par
The method of least squares is defined as follows (cf., e.g., \cite[p. 217]{gautschi2004orthogonal}): Let $U$ be a subspace of $\mathcal{C}\left[a,b\right]$ with
$\mbox{dim}\, U=n+1$. Let $x_\mu\in\left[a,b\right]$ be distinct nodes for $\mu=0,\dots,N$ and $N\geq n$. Let $w:\left[a,b\right]\rightarrow\mathbb{R}$ be a non-negative weight-function. The least square operator ${LS}_n^N:\mathcal{C}\left[a,b\right]\rightarrow U$ is defined by
\begin{equation}
\sum_{\mu=0}^N{\left({LS}_n^N[f]\left(x_\mu\right)-f\left(x_\mu\right)\right)^2 w\left(x_\mu\right)}=\min_{p\in U}{\sum_{\mu=0}^N{\left(p\left(x_\mu\right)-f\left(x_\mu\right)\right)^2 w\left(x_\mu\right)}}.
\end{equation}
Here we investigate the following problem: For which
\begin{itemize}
  \item functions $f\in K\subset\mathcal{C}\left[a,b\right]$,
	\item subspace $U\subset\mathcal{C}\left[a,b\right]$ with $\mbox{dim}\, U=n+1$,
	\item grid $\left\{x_0,\dots,x_N\right\}$ with $x_\mu\in\left[a,b\right]$,
	\item ratio $N/n$
\end{itemize}
do we have convergence of the sequence $\left({LS}_n^N[f]\right)$ for each $x\in\left[a,b\right]$?\\
In the following we consider the standard case
\begin{itemize}
  \item $\left[a,b\right]=\left[-1,1\right]$,
	\item $U=\mathcal{P}_n$ is the space of polynomials of degree $n$,
	\item $\left\{x_0,\dots,x_N\right\}$ is an equidistant grid with $N+1$ nodes on the interval $[-1,1]$, i. e. $x_\mu=-1+2\mu/N$ for $\mu=0,\dots,N$.
\end{itemize}
To investigate the above problem we describe the least square operator by the expansion of a function by Hahn polynomials. The Hahn polynomials $Q_k\equiv Q_k\left(\cdot;\alpha,\beta,N\right)$ are classical discrete orthogonal polynomials on the interval $I=\left[0,N\right]$ of degree $k$. They are orthogonal on $I$ with respect to the inner product
\begin{equation}
\left<f,g\right>_\omega:=\sum_{i=0}^N{f(i)g(i)\omega(i)},
\end{equation}
where $\omega$ is the weight-function given by
\begin{equation}
\omega(x):=\left(\begin{matrix}\alpha+x\\x\end{matrix}\right)\left(\begin{matrix}\beta+N-x\\N-x\end{matrix}\right).
\end{equation}
They are normalized by
\begin{equation}\label{orthoghahn1a}
\left<Q_k(\cdot;\alpha,\beta,N),Q_k(\cdot;\alpha,\beta,N)\right>_\omega=\frac{(-1)^k(k+\alpha+\beta+1)_{N+1}(\beta+1)_k k!}{(2k+\alpha+\beta+1)(\alpha+1)_k(-N)_k N!}
\end{equation}
(cf., e.g., \cite[p. 204]{koekoek2010hypergeometric}).\\
It is well-known that the least square operator ${LS}_n^N$ can be represented by use of Hahn polynomials (cf., e.g., \cite[p. 218-232]{wernerfunktionalanalysis}):
\begin{equation}
{LS}_n^N[f](x)=\sum_{k=0}^n{\frac{\left<f,Q_k\right>_\omega}{\left<Q_k,Q_k\right>_\omega}Q_k\left(\frac{N}{2}(1+x)\right)},
\end{equation}
where $f\in\mathcal{C}\left[-1,1\right]$.
\par
The following relation between Hahn polynomials and Jacobi polynomials $P_k\equiv P_k^{\alpha,\beta}$ is well-known.
\begin{equation}\label{s2}
\lim_{N\to\infty}{(-1)^k\binom{k+\alpha}{k}Q_k\left(\frac{N}{2}(1+x);\alpha,\beta,N\right)}=P_k^{\beta,\alpha}(x),
\end{equation}
for each $x\in[-1,1]$ (cf., e.g., \cite[p. 45]{nikiforov1991classical}). The Jacobi polynomials $P_k^{\alpha,\beta}$ are orthogonal on the interval $I=\left[-1,1\right]$ with respect to the inner product
\begin{equation}
(f,g)_\varrho:=\int_{-1}^1{f(x)g(x)\varrho(x)\mathrm{d}x},
\end{equation}
where
\begin{equation}
\varrho(x):=(1-x)^\alpha(1+x)^\beta
\end{equation}
is the weight-function. They are normalized by 
\begin{equation}\label{orthogjacobi1a}
\left(P_k^{\alpha,\beta},P_k^{\alpha,\beta}\right)_\varrho=\frac{2^{\alpha+\beta+1}\Gamma(k+\alpha+1)\Gamma(k+\beta+1)}{(2k+\alpha+\beta+1)k!\Gamma(k+\alpha+\beta+1)}
\end{equation}
(cf., e.g., \cite[p. 217]{koekoek2010hypergeometric}). By equation (\ref{s2}) the Hahn polynomials can be seen as the discrete analogue of the Jacobi polynomials.
\par
For several classes of functions the convergence of the corresponding expansions by use of Jacobi polynomials have been proved in the last decades. The following first main Theorem \ref{s16} shows the connection of these classical results with the problem stated above. We estimate the pointwise error of a function $u$ and their truncated Hahn series expansion with the pointwise error of $u$ with their truncated Jacobi series expansion.

\begin{theorem}\label{s16}
Let $\alpha\geq0$ and let for $N\in\N$
\begin{align}
n(\alpha,N):=\frac{1}{2}-\alpha+\frac{1}{2}\sqrt{(2\alpha+1)(2\alpha+2N+1)}.
\end{align}
For each $u\in\mathcal{C}\left[-1,1\right]\cap\mathcal{BV}\left[-1,1\right]$ holds
\begin{align}
\begin{aligned}
&\left\lvert u(x)-\sum_{k=0}^n{\frac{\left<u,Q_k\right>_\omega}{\left<Q_k,Q_k\right>_\omega}Q_k\left(\frac{N}{2}(1+x)\right)}\right\rvert\\
\leq&\left\lvert u(x)-\sum_{k=0}^n{\frac{\left(u,P_k\right)_\varrho}{\left(P_k,P_k\right)_\varrho}P_k(x)}\right\rvert+\mathcal{O}\left(\frac{n^{3+\alpha+\max{\{1,\alpha\}}}}{N}\right),
\end{aligned}
\end{align}
with $n\leq n(\alpha,N)$, for each $x\in[-1,1]$.
\end{theorem}

\begin{proof}
This result follows directly from Theorem \ref{s10} and Theorem \ref{s15} from section \ref{preliminaries}.
\end{proof}

Therefore with respect to the above problem we obtain directly from Theorem \ref{s16} our second main result:
\begin{theorem}\label{s17}
Let $\alpha\geq0$ and let for $N\in\N$
\begin{align}
n(\alpha,N):=\frac{1}{2}-\alpha+\frac{1}{2}\sqrt{(2\alpha+1)(2\alpha+2N+1)}.
\end{align}
For each $u\in\mathcal{C}\left[-1,1\right]\cap\mathcal{BV}\left[-1,1\right]$ and any sequence $(N_n)_{n\in\mathbb{N}}$ with $N_n^{-1}n^{3+\alpha+\max{\{1,\alpha\}}}\rightarrow0$ we have convergence of the least square method
\begin{equation}
\sum_{k=0}^n{\frac{\left<u,Q_k\right>_\omega}{\left<Q_k,Q_k\right>_\omega}Q_k\left(\frac{N}{2}(1+x)\right)}\rightarrow u(x),~n\rightarrow\infty
\end{equation}
for each $x\in[-1,1]$, for which the series of Jacobi polynomials converges to $u(x)$, i. e.
\begin{align}
\sum_{k=0}^n{\frac{\left(u,P_k\right)_\varrho}{\left(P_k,P_k\right)_\varrho}P_k(x)}\rightarrow u(x),~n\rightarrow\infty
\end{align}
for each $x\in[-1,1]$.
\end{theorem}

Therefore the series expansion $\sum_{k=0}^n{\hat{u} Q_k}$ of a function $u$ by Hahn polynomials $Q_k$ converges pointwise, if the series expansion $\sum_{k=0}^n{\hat{u} P_k}$ of the function $u$ by Jacobi polynomials $P_k$ converges pointwise and if $n^{3+\alpha+\max{\{1,\alpha\}}}/N\rightarrow0$ for $n,N\rightarrow\infty$.
\par
In the special case $\alpha=0$ the Jacobi polynomials reduce to the Legendre polynomials. Using Jackson's Theorem:

\begin{lemma}[cf. \cite{jackson1930theory}]
Let $I=[-1,1]$ be a compact interval. For each function $f\in\mathcal{C}^1\left[-1,1\right]$, with $f^\prime\in\mathcal{BV}\left[-1,1\right]$, the series of the Legendre polynomials converges uniformly to $f$ in $I$.
\end{lemma}

We have the following interesting special answer of the Problem stated above:\\
For which $f$, i. e. for which function classes $K\subset\mathcal{C}\left[-1,1\right]$ and for which ratio $N/n$ do we have the convergence of the least square method ${LS}_n^N$ for each $x\in\left[-1,1\right]$?

\begin{korollar}\label{s18}
Let $f\in K:=\left\{g\in\mathcal{C}^1\left[-1,1\right]:g^\prime\in\mathcal{BV}\left[-1,1\right]\right\}$ and let $(N_n)_{n\in\mathbb{N}}$ be a sequence with 
\begin{equation}
\lim_{n\rightarrow\infty}{\frac{n^4}{N_n}}=0.
\end{equation}
Then the least square method ${LS}_n^{N_n}[f]$ converges for each $x\in[-1,1]$.
\end{korollar}

Finally, we describe our motivation for further application of our result:
\begin{anmerk}\label{motivation}
There are many other reasons to consider the truncated Hahn series. Continuous and discrete orthogonal polynomials are often used in applied mathematics like combinatorics, number theory, statistics, mathematical physics (cf., e.g., \cite{Karlin:02,nikiforov1991classical}). Also many numerical methods for ordinary or partial differential equations use series expansion of a function $f$ in sequence of orthogonal polynomials like Chebyshev, Legendre or Hermite polynomials, for details (cf., e.g., \cite{funaro2008polynomial, Canuto:01, ranocha2016summation, pettersson2009numerical}) and references therein. In the past normally continuous orthogonal polynomials are mainly utilized, but nowaday motivated by stochastic differential equations many researchers consider and apply discrete orthogonal polynomials (cf., e.g., \cite{xiu2002wiener, pettersson2015polynomial}).\\
But this is not our motivation why we are interested in the Hahn expansion and there behavior. Our interest is due to the fact, that we want to apply these polynomials in a spectral method and/or to construct a discrete filter from the eigenvalue equation similar to \cite{glaubitz2016application}. Therefore we already employ with the Hahn coefficients of a function $f$ and study in \cite{Goertz:02}, how fast the absolute values of the coefficients tend to zero compared to others like Legendre coefficients. The Hahn coefficients decay rapidly, but nevertheless further problems arises by using the Hahn expansion. The truncated series expansion may not describe qualitatively the original function in the pointwise sense, see \cite{Goertz:02}. Therefore we are interested in the pointwise error between a function $f$ and the truncated Hahn series.
\end{anmerk}

\section{Preliminaries}
\label{preliminaries}

In this section we demonstrate preliminary Lemmatas and prove our main Theorem \ref{s16}. The first step on the way to solve the above problem is to calculate a ratio $N/n$ for the convergence in equation (\ref{s2}). In equation (\ref{s2}) is $n$ fixed. But what happens, if $n$ and $N$ increase at the same time. The next Lemma give us a ratio $N/n$ for the boundedness of the Hahn polynomials.

\begin{lemma}\label{s1}
Let $\alpha>-\frac{1}{2}$ and let for $N\in\N$
\begin{align}
n(\alpha,N):=\frac{1}{2}-\alpha+\frac{1}{2}\sqrt{(2\alpha+1)(2\alpha+2N+1)}.
\end{align}
Then, for any $n\leq n(\alpha,N)$ holds
\begin{align}\label{s1a}
\max_{x\in[0,N]}{\left\lvert Q_n(x;\alpha,\alpha,N)\right\rvert}=1.
\end{align}
\end{lemma}

\begin{proof}
This result follows directly from \cite[Theorem 3.2.]{2}.
\end{proof}

\begin{anmerk}\label{maxjacobi1}
Let $\alpha>-\frac{1}{2}$. Then one has
\begin{align}\label{maxjacobi1a}
\max_{x\in[-1,1]}{\left\lvert P_k^{\alpha,\alpha}(x)\right\rvert}=\binom{k+\alpha}{k}
\end{align}
(cf., e.g., \cite[p. 786]{4}).
\end{anmerk}

We now consider the recurrences, that satisfy the Jacobi polynomials and the Hahn polynomials. After change the equations, we can compare the coefficients. With the uniqueness of the solutions of the recurrences we obtain the ratio $N/n$ for the convergence in equation (\ref{s2}).

\begin{anmerk}\label{rekurjacobi1}
The Jacobi polynomials satisfy the following recurrence
\begin{align}
x P_k^{\alpha,\beta}(x)=\alpha_k P_{k+1}^{\alpha,\beta}(x)+\beta_k P_k^{\alpha,\beta}(x)+\gamma_k P_{k-1}^{\alpha,\beta}(x),\ \mbox{for}\ x\in[-1,1]
\end{align}
with the constants
\begin{align}
\alpha_k&=\frac{2(k+1)(k+\alpha+\beta+1)}{(2k+\alpha+\beta+1)(2k+\alpha+\beta+2)},\\
\beta_k&=\frac{\beta^2-\alpha^2}{(2k+\alpha+\beta)(2k+\alpha+\beta+2)},\\
\gamma_k&=\frac{2(k+\alpha)(k+\beta)}{(2k+\alpha+\beta)(2k+\alpha+\beta+1)},
\end{align}
(cf., e.g., \cite[p. 217]{koekoek2010hypergeometric}).
\end{anmerk}

\begin{anmerk}\label{rekurhahn1}
The Hahn polynomials satisfy the following recurrence
\begin{align}\label{rekurhahn1a}
-x Q_k(x)=A_k Q_{k+1}(x)-\left(A_k+C_k\right)Q_k(x)+C_k Q_{k-1}(x),
\end{align}
for $x\in[0,N]$ with the constants
\begin{align}
A_k&=\frac{(k+\alpha+\beta+1)(k+\alpha+1)(N-k)}{(2k+\alpha+\beta+1)(2k+\alpha+\beta+2)},\\
C_k&=\frac{k(k+\alpha+\beta+N+1)(k+\beta)}{(2k+\alpha+\beta)(2k+\alpha+\beta+1)},
\end{align}
(cf., e.g., \cite[p. 204-205]{koekoek2010hypergeometric}).
\end{anmerk}

\begin{anmerk}\label{rekurhahn2}
For the sake of brevity we introduce
\begin{align}
\tilde{Q}_k(x):=(-1)^k\binom{k+\alpha}{k}Q_k\left(\frac{N}{2}(1+x);\alpha,\alpha,N\right).
\end{align}
By substitution in (\ref{rekurhahn1a}) we obtain for $\tilde{Q}_k$ with $x\in[-1,1]$ the recurrence
\begin{align*}
&\frac{N}{2}(1+x)\tilde{Q}_k(x)\\
=&A_k\frac{\binom{k+\alpha}{k}}{\binom{k+1+\alpha}{k+1}}\tilde{Q}_{k+1}(x)+\left(A_k+C_k\right)\tilde{Q}_k(x)+C_k\frac{\binom{k+\alpha}{k}}{\binom{k-1+\alpha}{k-1}}\tilde{Q}_{k-1}(x)\\
=&A_k\frac{k+1}{k+1+\alpha}\tilde{Q}_{k+1}(x)+\left(A_k+C_k\right)\tilde{Q}_k(x)+C_k\frac{k+\alpha}{k}\tilde{Q}_{k-1}(x).
\end{align*}
Change the equation leads to
\begin{align*}
&x\tilde{Q}_k(x)\\
=&\frac{2}{N}A_k\frac{k+1}{k+1+\alpha}\tilde{Q}_{k+1}(x)+\frac{2}{N}\left(A_k+C_k\right)\tilde{Q}_k(x)-\tilde{Q}_k(x)+\frac{2}{N}C_k\frac{k+\alpha}{k}\tilde{Q}_{k-1}(x)\\
=&\frac{k+1}{k+1+\alpha}\frac{k+2\alpha+1}{2k+2\alpha+1}\left(1-\frac{k}{N}\right)\tilde{Q}_{k+1}(x)+\frac{k+\alpha}{k}\frac{k(k+2\alpha+N+1)}{N(2k+2\alpha+1)}\tilde{Q}_{k-1}(x)\\
=&\alpha_k\left(1-\frac{k}{N}\right)\tilde{Q}_{k+1}(x)+\gamma_k\left(1+\frac{k+2\alpha+1}{N}\right)\tilde{Q}_{k-1}(x).
\end{align*}
\end{anmerk}
So we obtain the relations in the following Lemma by using the Lemma \ref{s1} and the equation (\ref{s2}).

\begin{lemma}\label{s3}
Let $\alpha>-\frac{1}{2}$ and let for $N\in\N$
\begin{align}
n(\alpha,N):=\frac{1}{2}-\alpha+\frac{1}{2}\sqrt{(2\alpha+1)(2\alpha+2N+1)}.
\end{align}
Then one has the following relations:
\begin{align}\label{s3a}
(-1)^n\binom{n+\alpha}{n}Q_n\left(\frac{N}{2}(1+x);\alpha,\alpha,N\right)=P_n^{\alpha,\alpha}(x)+\mathcal{O}\left(\frac{n^2}{N}\right),
\end{align}
with $n\leq n(\alpha,N)$, for $x\in[-1,1]$ and $-\frac{1}{2}<\alpha<1$.
\begin{align}\label{s3b}
(-1)^n\binom{n+\alpha}{n}Q_n\left(\frac{N}{2}(1+x);\alpha,\alpha,N\right)=P_n^{\alpha,\alpha}(x)+\mathcal{O}\left(\frac{n^{1+\alpha}}{N}\right),
\end{align}
with $n\leq n(\alpha,N)$, for $x\in[-1,1]$ and $1\leq\alpha$.
\end{lemma}

The following Lemmatas \ref{s4} - \ref{trapezformel} are technical and preliminary Lemmatas for the next steps. They give us some properties of Jacobi polynomials and Hahn polynomials.

\begin{lemma}\label{s4}
One has
\begin{align}\label{s4a}
\begin{aligned}
&\left<Q_n(\cdot;\alpha,\beta,N),Q_n(\cdot;\alpha,\beta,N)\right>_\omega\\
=&\frac{n!n!\Gamma(n+\alpha+\beta+2+N)\Gamma(\alpha+1)(N-n)!}{\Gamma(\beta+1)\Gamma(1+\alpha+n)\Gamma(1+\alpha+n)N!N!2^{\alpha+\beta+1}}\left(P_n^{\alpha,\beta},P_n^{\alpha,\beta}\right)_\varrho.
\end{aligned}
\end{align}
\end{lemma}

\begin{proof}
With the equations (\ref{orthoghahn1a}) and (\ref{orthogjacobi1a}) we have
\begin{align*}
&\frac{\left<Q_n(\cdot;\alpha,\beta,N),Q_n(\cdot;\alpha,\beta,N)\right>_\omega}{\left(P_n^{\alpha,\beta},P_n^{\alpha,\beta}\right)_\varrho}\\
=&\frac{(-1)^n(n+\alpha+\beta+1)_{N+1}(\beta+1)_n n!}{(2n+\alpha+\beta+1)(\alpha+1)_n(-N)_n N!}\cdot\frac{(2n+\alpha+\beta+1)n!\Gamma(n+\alpha+\beta+1)}{2^{\alpha+\beta+1}\Gamma(n+\alpha+1)\Gamma(n+\beta+1)}\\
=&\frac{n!n!\Gamma(n+\alpha+\beta+2+N)\Gamma(\alpha+1)(N-n)!}{\Gamma(\beta+1)\Gamma(1+\alpha+n)\Gamma(1+\alpha+n)N!N!2^{\alpha+\beta+1}}.
\end{align*}
\end{proof}

\begin{lemma}\label{gamma1}[(cf., e.g., \cite[p. 257]{4})]
For $a, b>0$ holds
\begin{align}\label{gamma1a}
N^{b-a}\frac{\Gamma(N+a)}{\Gamma(N+b)}=1+\frac{(a-b)(a+b-1)}{2N}+\mathcal{O}\left(\frac{1}{N^2}\right).
\end{align}
\end{lemma}

\begin{lemma}\label{s5}
For $\alpha=\beta$ holds
\begin{align}\label{s5a}
\omega\left(\frac{N}{2}(1+x)\right)=\frac{N^{2\alpha}}{2^{2\alpha}\Gamma(\alpha+1)\Gamma(\alpha+1)}\left(1+\mathcal{O}\left(\frac{1}{N}\right)\right)\varrho(x),
\end{align}
for $x\in[-1,1]$.
\end{lemma}

\begin{proof}
For $\alpha=\beta$ and each $x\in[-1,1]$ holds
\begin{align*}
&\frac{\omega\left(\frac{N}{2}(1+x)\right)}{\varrho(x)}\\
=&\binom{\alpha+\frac{N}{2}(1+x)}{\frac{N}{2}(1+x)}\binom{\alpha+N-\frac{N}{2}(1+x)}{N-\frac{N}{2}(1+x)}(1-x)^{-\alpha}(1+x)^{-\alpha}\\
=&\frac{\Gamma\left(\alpha+\frac{N}{2}(1+x)+1\right)\Gamma\left(\alpha+N-\frac{N}{2}(1+x)+1\right)(1-x)^{-\alpha}(1+x)^{-\alpha}}{\Gamma\left(\frac{N}{2}(1+x)+1\right)\Gamma\left(\alpha+1\right)\Gamma\left(N-\frac{N}{2}(1+x)+1\right)\Gamma\left(\alpha+1\right)}\\
=&\frac{\Gamma\left(\alpha+1+\frac{N}{2}(1+x)\right)\Gamma\left(\alpha+1+\frac{N}{2}(1-x)\right)}{\Gamma\left(\frac{N}{2}(1+x)+1\right)\Gamma\left(\alpha+1\right)\Gamma\left(\frac{N}{2}(1-x)+1\right)\Gamma\left(\alpha+1\right)(1-x)^{\alpha}(1+x)^{\alpha}}\\
\overset{(\ref{gamma1a})}{=}&\frac{\left(\frac{N}{2}(1+x)\right)^\alpha\left(\frac{N}{2}(1-x)\right)^\alpha}{\Gamma\left(\alpha+1\right)\Gamma\left(\alpha+1\right)(1-x)^{\alpha}(1+x)^{\alpha}}\left(1+\mathcal{O}\left(\frac{1}{N}\right)\right)\\
=&\frac{N^{2\alpha}}{2^{2\alpha}\Gamma(\alpha+1)\Gamma(\alpha+1)}\left(1+\mathcal{O}\left(\frac{1}{N}\right)\right).
\end{align*}
\end{proof}

In the next Lemma, we consider the total variation of the Jacobi polynomials. The total variation of a real-valued function $f:[a,b]\rightarrow\mathbb{R}$ is the value
\begin{align}
\mathcal{V}_a^b\left[f\right]=\sup_{P\in\mathcal{P}}{\sum_{i=0}^{n-1}{\left\lvert f\left(x_{i+1}\right)-f\left(x_i\right)\right\rvert}},
\end{align}
whereby the supremum is taken over the set
\begin{align}
\mathcal{P}=\left\{P=\left\{x_0,\dots,x_n\right\}:\ a=x_0<x_1<\dots<x_{n-1}<x_n=b,\ n\in\mathbb{N}\right\}
\end{align}
of all partitions of the interval $[a,b]$ (cf., e.g., \cite[p. 493]{7}). A real-valued function $f:[a,b]\rightarrow\mathbb{R}$ is said to be of bounded variation, if its total variation is finite. We define the class of functions of bounded variation
\begin{align}
\mathcal{BV}\left[a,b\right]=\left\{f:[a,b]\rightarrow\mathbb{R}:\ \exists C>0:\mathcal{V}_a^b\left[f\right]<C\right\}.
\end{align}
For the sake of brevity we introduce $\mathcal{V}\left[f\right]\equiv\mathcal{V}_{-1}^1\left[f\right]$.

\begin{lemma}\label{totvarjacobi1}
For $\alpha\geq0$ we have
\begin{align}\label{totvarjacobi1a}
\mathcal{V}\left[P_n^{\alpha,\alpha}\right]\leq 2n\binom{n+\alpha}{n}.
\end{align}
\end{lemma}

\begin{proof}
Let $\alpha\geq0$. We have for the derivative
\begin{align}
\frac{\mathrm{d}}{\mathrm{d}x}P_n^{\alpha,\alpha}(x)=\frac{2\alpha+n+1}{2}P_{n-1}^{\alpha+1,\alpha+1}(x)
\end{align}
(cf., e.g., \cite[p. 10]{nikiforov1991classical}). The polynomial $P_{n-1}^{\alpha+1,\alpha+1}$ has exactly $n-1$ simple zeros $z_i$, $i=1,\dots,n-1$ in the interval $(-1,1)$ (cf., e.g., \cite[p. 15]{nikiforov1991classical}). Thereby the polynomial $P_n^{\alpha,\alpha}$ has exactly $n-1$ local extrema in the interval $(-1,1)$. So is $P_n^{\alpha,\alpha}$ monotonous on the $n$ subintervals $[-1,z_1]$, $[z_i,z_{i+1}]$, $i=1,\dots,n-2$ and $[z_{n-1},1]$. Let $z_0:=-1$ and let $z_n:=1$. With Remark \ref{maxjacobi1} we have
\begin{align}
\max_{x\in[-1,1]}{\left\lvert P_n^{\alpha,\alpha}(x)\right\rvert}=\binom{n+\alpha}{n}
\end{align}
and therefore we obtain
\begin{align*}
\mathcal{V}\left[P_n^{\alpha,\alpha}\right]&\leq\sum_{i=0}^{n-1}{\mathcal{V}\left[P_n^{\alpha,\alpha}\vert_{[z_i,z_{i+1}]}\right]} \\
&\leq\sum_{i=0}^{n-1}{\left\lvert P_n^{\alpha,\alpha}(z_{i+1})-P_n^{\alpha,\alpha}(z_i)\right\rvert} \\
&\leq\sum_{i=0}^{n-1}{\left\lvert P_n^{\alpha,\alpha}(z_{i+1})\right\rvert+\left\lvert P_n^{\alpha,\alpha}(z_i)\right\rvert} \\
&\leq\sum_{i=0}^{n-1}{2\max_{x\in[-1,1]}{\left\lvert P_n^{\alpha,\alpha}(x)\right\rvert}} \\
&\leq 2n\max_{x\in[-1,1]}{\left\lvert P_n^{\alpha,\alpha}(x)\right\rvert} \\
&\leq 2n\binom{n+\alpha}{n}.
\end{align*}
\end{proof}

\begin{lemma}\label{totvarjacobi2}
For $\alpha>0$ one has
\begin{align}
\mathcal{V}\left[P_n^{\alpha,\alpha} \varrho\right]\leq 2(n+1)\max_{x\in[-1,1]}{\left\lvert P_n^{\alpha,\alpha}(x)\varrho(x)\right\rvert}.
\end{align}
\end{lemma}

\begin{proof}
Let $\alpha>0$. We have for the derivative (cf., e.g., \cite[p. 10]{nikiforov1991classical})
\begin{align*}
&\frac{\mathrm{d}}{\mathrm{d}x}\left[P_n^{\alpha,\alpha} \varrho\right](x)\\
=&\left(1-x^2\right)^\alpha\frac{\mathrm{d}}{\mathrm{d}x}P_n^{\alpha,\alpha}(x)+P_n^{\alpha,\alpha}(x)\frac{\mathrm{d}}{\mathrm{d}x}\left(1-x^2\right)^\alpha \\
=&\left(1-x^2\right)^\alpha\frac{2\alpha+n+1}{2}P_{n-1}^{\alpha+1,\alpha+1}(x)-P_n^{\alpha,\alpha}(x)2x\alpha\left(1-x^2\right)^{\alpha-1} \\
=&\left(1-x^2\right)^{\alpha-1}\left[\frac{2\alpha+n+1}{2}\left(1-x^2\right)P_{n-1}^{\alpha+1,\alpha+1}(x)-2x\alpha P_n^{\alpha,\alpha}(x)\right] \\
=:&\left(1-x^2\right)^{\alpha-1}\tilde{P}(x).
\end{align*}
Since $\tilde{P}\in\mathcal{P}_{n+1}$ is a polynomial of degree $n+1$, it has at most $n+1$ zeros $z_i$, $i=1,\dots,k$, $k\leq n+1$ in the interval $(-1,1)$. Thereby $P_n^{\alpha,\alpha} \varrho$ has at most $k\leq n+1$ local extrema in the interval $(-1,1)$. Then $P_n^{\alpha,\alpha} \varrho$ is monotonous on the $k+1$ subintervals $[-1,z_1]$, $[z_i,z_{i+1}]$, $i=1,\dots,k-1$ and $[z_k,1]$. Let $z_0:=-1$ and let $z_{k+1}:=1$. Then one has
\begin{align*}
\mathcal{V}\left[P_n^{\alpha,\alpha} \varrho\right]&\leq\sum_{i=0}^k{\mathcal{V}\left[P_n^{\alpha,\alpha} \varrho\vert_{[z_i,z_{i+1}]}\right]} \\
&\leq\sum_{i=0}^k{\left\lvert P_n^{\alpha,\alpha}(z_{i+1})\varrho(z_{i+1})-P_n^{\alpha,\alpha}(z_i)\varrho(z_i)\right\rvert} \\
&\leq\sum_{i=0}^k{\left\lvert P_n^{\alpha,\alpha}(z_{i+1})\varrho(z_{i+1})\right\rvert+\left\lvert P_n^{\alpha,\alpha}(z_i)\varrho(z_i)\right\rvert} \\
&\leq\left\lvert P_n^{\alpha,\alpha}(z_0)\varrho(z_0)\right\rvert+\left\lvert P_n^{\alpha,\alpha}(z_{k+1})\varrho(z_{k+1})\right\rvert+\sum_{i=1}^k{2\left\lvert P_n^{\alpha,\alpha}(z_i)\varrho(z_i)\right\rvert} \\
&\leq\left\lvert P_n^{\alpha,\alpha}(-1)\varrho(-1)\right\rvert+\left\lvert P_n^{\alpha,\alpha}(1)\varrho(1)\right\rvert+\sum_{i=1}^k{2\max_{x\in[-1,1]}{\left\lvert P_n^{\alpha,\alpha}(x)\varrho(x)\right\rvert}} \\
&\leq\sum_{i=1}^k{2\max_{x\in[-1,1]}{\left\lvert P_n^{\alpha,\alpha}(x)\varrho(x)\right\rvert}} \\
&\leq 2k\max_{x\in[-1,1]}{\left\lvert P_n^{\alpha,\alpha}(x)\varrho(x)\right\rvert} \\
&\leq 2(n+1)\max_{x\in[-1,1]}{\left\lvert P_n^{\alpha,\alpha}(x)\varrho(x)\right\rvert}.
\end{align*}
\end{proof}

\begin{lemma}\label{totvarjacobi3}
For $\alpha\geq\frac{1}{2}$ one has
\begin{align}\label{totvarjacobi3a}
\max_{x\in[-1,1]}{\left\lvert P_n^{\alpha,\alpha}(x)\varrho(x)\right\rvert}\leq\left\lvert \frac{\Gamma\left(2\alpha+1\right)\Gamma\left(n+\alpha+1\right)\Gamma\left(\frac{n}{2}+\alpha+\frac{1}{2}\right)}{\Gamma\left(\alpha+1\right)\Gamma\left(\alpha+\frac{1}{2}\right)\Gamma\left(n+2\alpha+1\right)\Gamma\left(\frac{n}{2}+1\right)}\right\rvert.
\end{align}
For $\frac{1}{2}\geq\alpha>0$ one has
\begin{align}\label{totvarjacobi3b}
\max_{x\in[-1,1]}{\left\lvert P_n^{\alpha,\alpha}(x)\varrho(x)\right\rvert}\leq\left\lvert \frac{\Gamma\left(\alpha\right)\Gamma\left(2\alpha+1\right)\Gamma\left(n+\alpha+1\right)}{\sqrt{\pi}\Gamma\left(\alpha+\frac{1}{2}\right)\Gamma\left(\alpha+1\right)\Gamma\left(n+2\alpha+1\right)}\right\rvert.
\end{align}
\end{lemma}

\begin{proof}
We use the Gegenbauer polynomials
\begin{align}\label{gegenbauer1a}
P_n^\lambda(x):=\frac{\Gamma\left(\lambda+\frac{1}{2}\right)\Gamma\left(n+2\lambda\right)}{\Gamma\left(2\lambda\right)\Gamma\left(n+\lambda+\frac{1}{2}\right)}P_n^{\lambda-\frac{1}{2},\lambda-\frac{1}{2}}(x),
\end{align}
(cf., e.g., \cite[p. 80]{3}) and the estimations
\begin{align}\label{gegenbauer1b}
\left[P_n^\lambda(x)\right]^2\leq\left(1-x^2\right)^{1-2\lambda}\left[\frac{\Gamma\left(\frac{n}{2}+\lambda\right)}{\Gamma(\lambda)\Gamma\left(\frac{n}{2}+1\right)}\right]^2,~\mbox{for}~\lambda\geq1,
\end{align}
and
\begin{align}\label{gegenbauer1c}
\left[P_n^\lambda(x)\right]^2\leq\frac{1}{\pi}\left[\frac{\Gamma\left(\lambda-\frac{1}{2}\right)}{\Gamma(\lambda)}\right]^2\left(1-x^2\right)^{1-2\lambda},~\mbox{for}~\frac{1}{2}<\lambda\leq1,
\end{align} 
(cf., e.g., \cite[p. 362-363]{6}). You find further results in \cite{forster1993inequalities}. Let for the moment $\alpha\geq\frac{1}{2}$ and let $\lambda:=\alpha+\frac{1}{2}$, then holds $\lambda\geq1$. Let $x\in[-1,1]$, then we have
\begin{align*}
\left\lvert P_n^{\alpha,\alpha}(x)\varrho(x)\right\rvert&=\left\lvert P_n^{\lambda-\frac{1}{2},\lambda-\frac{1}{2}}(x)\left(1-x^2\right)^{\lambda-\frac{1}{2}}\right\rvert \\
&\overset{(\ref{gegenbauer1a})}{=}\left\lvert \frac{\Gamma\left(2\lambda\right)\Gamma\left(n+\lambda+\frac{1}{2}\right)}{\Gamma\left(\lambda+\frac{1}{2}\right)\Gamma\left(n+2\lambda\right)}P_n^\lambda(x)\left(1-x^2\right)^{\lambda-\frac{1}{2}}\right\rvert \\
&\overset{(\ref{gegenbauer1b})}{\leq}\left\lvert \frac{\Gamma\left(2\lambda\right)\Gamma\left(n+\lambda+\frac{1}{2}\right)}{\Gamma\left(\lambda+\frac{1}{2}\right)\Gamma\left(n+2\lambda\right)}\frac{\Gamma\left(\frac{n}{2}+\lambda\right)}{\Gamma(\lambda)\Gamma\left(\frac{n}{2}+1\right)}\right\rvert \\
&\leq\left\lvert \frac{\Gamma\left(2\alpha+1\right)\Gamma\left(n+\alpha+1\right)}{\Gamma\left(\alpha+1\right)\Gamma\left(n+2\alpha+1\right)}\frac{\Gamma\left(\frac{n}{2}+\alpha+\frac{1}{2}\right)}{\Gamma\left(\alpha+\frac{1}{2}\right)\Gamma\left(\frac{n}{2}+1\right)}\right\rvert.
\end{align*}
The right-hand side is independent of $x$, therefore one has (\ref{totvarjacobi3a}).\\
Now let $\frac{1}{2}\geq\alpha>0$ and let $\lambda:=\alpha+\frac{1}{2}$, then holds $\frac{1}{2}<\lambda\leq1$. Let $x\in[-1,1]$, then we have
\begin{align*}
\left\lvert P_n^{\alpha,\alpha}(x)\varrho(x)\right\rvert&=\left\lvert P_n^{\lambda-\frac{1}{2},\lambda-\frac{1}{2}}(x)\left(1-x^2\right)^{\lambda-\frac{1}{2}}\right\rvert \\
&\overset{(\ref{gegenbauer1a})}{=}\left\lvert \frac{\Gamma\left(2\lambda\right)\Gamma\left(n+\lambda+\frac{1}{2}\right)}{\Gamma\left(\lambda+\frac{1}{2}\right)\Gamma\left(n+2\lambda\right)}P_n^\lambda(x)\left(1-x^2\right)^{\lambda-\frac{1}{2}}\right\rvert \\
&\overset{(\ref{gegenbauer1c})}{\leq}\left\lvert \frac{\Gamma\left(2\lambda\right)\Gamma\left(n+\lambda+\frac{1}{2}\right)}{\Gamma\left(\lambda+\frac{1}{2}\right)\Gamma\left(n+2\lambda\right)}\frac{1}{\sqrt{\pi}}\frac{\Gamma\left(\lambda-\frac{1}{2}\right)}{\Gamma(\lambda)}\right\rvert \\
&\leq\left\lvert \frac{\Gamma\left(2\alpha+1\right)\Gamma\left(n+\alpha+1\right)}{\Gamma\left(\alpha+1\right)\Gamma\left(n+2\alpha+1\right)}\frac{1}{\sqrt{\pi}}\frac{\Gamma\left(\alpha\right)}{\Gamma\left(\alpha+\frac{1}{2}\right)}\right\rvert.
\end{align*}
The right-hand side is independent of $x$, therefore one has (\ref{totvarjacobi3b}).
\end{proof}

\begin{korollar}\label{totvarjacobi4}
For $\alpha\geq\frac{1}{2}$ we have
\begin{align}\label{totvarjacobi4a}
\max_{x\in[-1,1]}{\left\lvert P_n^{\alpha,\alpha}(x)\varrho(x)\right\rvert}=\mathcal{O}\left(n^{-\frac{1}{2}}\right).
\end{align}
For $\frac{1}{2}\geq\alpha>0$ we have
\begin{align}\label{totvarjacobi4b}
\max_{x\in[-1,1]}{\left\lvert P_n^{\alpha,\alpha}(x)\varrho(x)\right\rvert}=\mathcal{O}\left(n^{-\alpha}\right).
\end{align}
\end{korollar}

\begin{proof}
Let for the moment $\alpha\geq\frac{1}{2}$. With Lemma \ref{totvarjacobi3} one has
\begin{align*}
\max_{x\in[-1,1]}{\left\lvert P_n^{\alpha,\alpha}(x)\varrho(x)\right\rvert}&\leq\left\lvert \frac{\Gamma\left(2\alpha+1\right)\Gamma\left(n+\alpha+1\right)\Gamma\left(\frac{n}{2}+\alpha+\frac{1}{2}\right)}{\Gamma\left(\alpha+1\right)\Gamma\left(\alpha+\frac{1}{2}\right)\Gamma\left(n+2\alpha+1\right)\Gamma\left(\frac{n}{2}+1\right)}\right\rvert \\
&\overset{(\ref{gamma1a})}{=}\mathcal{O}\left(n^{\alpha+1-(2\alpha+1)}\right)\mathcal{O}\left(\left(\frac{n}{2}\right)^{\alpha+\frac{1}{2}-1}\right) \\
&=\mathcal{O}\left(n^{-\frac{1}{2}}\right).
\end{align*}
Now let $\frac{1}{2}\geq\alpha>0$. With Lemma \ref{totvarjacobi3} one has
\begin{align*}
\max_{x\in[-1,1]}{\left\lvert P_n^{\alpha,\alpha}(x)\varrho(x)\right\rvert}&\leq\left\lvert \frac{\Gamma\left(\alpha\right)\Gamma\left(2\alpha+1\right)\Gamma\left(n+\alpha+1\right)}{\sqrt{\pi}\Gamma\left(\alpha+\frac{1}{2}\right)\Gamma\left(\alpha+1\right)\Gamma\left(n+2\alpha+1\right)}\right\rvert \\
&\overset{(\ref{gamma1a})}{=}\mathcal{O}\left(n^{\alpha+1-(2\alpha+1)}\right) \\
&=\mathcal{O}\left(n^{-\alpha}\right).
\end{align*}
\end{proof}

\begin{korollar}\label{totvarjacobi5}
For $\alpha\geq\frac{1}{2}$ one has
\begin{align}\label{totvarjacobi5a}
\mathcal{V}\left[P_n^{\alpha,\alpha} \varrho\right]=\mathcal{O}\left(n^{\frac{1}{2}}\right).
\end{align}
For $\frac{1}{2}\geq\alpha>0$ one has
\begin{align}\label{totvarjacobi5b}
\mathcal{V}\left[P_n^{\alpha,\alpha} \varrho\right]=\mathcal{O}\left(n^{1-\alpha}\right).
\end{align}
\end{korollar}

\begin{proof}
Let for the moment $\alpha\geq\frac{1}{2}$. With Lemma \ref{totvarjacobi2} one has
\begin{align*}
\mathcal{V}\left[P_n^{\alpha,\alpha} \varrho\right]\leq 2(n+1)\max_{x\in[-1,1]}{\left\lvert P_n^{\alpha,\alpha}(x)\varrho(x)\right\rvert}.
\end{align*}
Also holds with Corollary \ref{totvarjacobi4}
\begin{align*}
\max_{x\in[-1,1]}{\left\lvert P_n^{\alpha,\alpha}(x)\varrho(x)\right\rvert}=\mathcal{O}\left(n^{-\frac{1}{2}}\right).
\end{align*}
Then we have (\ref{totvarjacobi5a}).\\
Now let $\frac{1}{2}\geq\alpha>0$. With Lemma \ref{totvarjacobi2} one has
\begin{align*}
\mathcal{V}\left[P_n^{\alpha,\alpha} \varrho\right]\leq 2(n+1)\max_{x\in[-1,1]}{\left\lvert P_n^{\alpha,\alpha}(x)\varrho(x)\right\rvert}.
\end{align*}
Also holds with Corollary \ref{totvarjacobi4}
\begin{align*}
\max_{x\in[-1,1]}{\left\lvert P_n^{\alpha,\alpha}(x)\varrho(x)\right\rvert}=\mathcal{O}\left(n^{-\alpha}\right).
\end{align*}
Then we have (\ref{totvarjacobi5b}).
\end{proof}

\begin{lemma}\label{s6}
For $n\leq\frac{N}{2}$ holds
\begin{align}
\left\lvert 1-\frac{N!N!N}{\Gamma(n+2+N)(N-n)!}\right\rvert\leq\frac{1+2n^2}{N+1}.
\end{align}
\end{lemma}

\begin{proof}
$n=0$ is trivial. For any $1\leq n\leq\frac{N}{2}$ we have
\begin{align*}
\left\lvert 1-\frac{N!N!N}{\Gamma(n+2+N)(N-n)!}\right\rvert&=\left\lvert 1-\frac{N}{N+1}\frac{N!(N+1)!}{(n+1+N)!(N-n)!}\right\rvert\\
&=\left\lvert 1-\frac{N}{N+1}\prod_{k=0}^{n-1}{\frac{N-k}{N+k+2}}\right\rvert\\
&=1-\frac{N}{N+1}\prod_{k=0}^{n-1}{\frac{N-k}{N+k+2}}\\
&\leq 1-\frac{N}{N+1}\prod_{k=0}^{n-1}{\frac{N-n+1}{N+n+1}}\\
&=1-\frac{N}{N+1}\left(\frac{N-n+1}{N+n+1}\right)^n\\
&\leq 1-\frac{N}{N+1}\left(\frac{N-2n}{N}\right)^n\\
&=1-\frac{N}{N+1}\left(1+\frac{-2n}{N}\right)^n.
\end{align*}
We now use the Bernoulli Inequality
\begin{align}
\left(1+x\right)^m\geq 1+mx,
\end{align}
for all $m\in\mathbb{N}$ and $x\geq -1$ (cf., e.g., \cite[p. 62]{7}),
\begin{align*}
\left\lvert 1-\frac{N!N!N}{\Gamma(n+2+N)(N-n)!}\right\rvert&\leq 1-\frac{N}{N+1}\left(1+n\frac{-2n}{N}\right)\\
&=\frac{1}{N+1}+2\frac{n^2}{N+1}.
\end{align*}
\end{proof}

\begin{korollar}\label{s6a}
For $n\leq\frac{N}{2}$ holds
\begin{align}\label{s6aa}
\frac{N!N!N}{\Gamma(n+2+N)(N-n)!}=1+\mathcal{O}\left(\frac{n^2}{N}\right).
\end{align}
\end{korollar}

\begin{proof}
With Lemma \ref{s6} holds for any $n\leq\frac{N}{2}$
\begin{align*}
\left\lvert 1-\frac{N!N!N}{\Gamma(n+2+N)(N-n)!}\right\rvert\leq\frac{1+2n^2}{N+1}=\mathcal{O}\left(\frac{n^2}{N}\right).
\end{align*}
\end{proof}

\begin{lemma}\label{bvprod}
For each $f,g\in\mathcal{BV}\left[-1,1\right]$ we have
\begin{align}\label{bvproda}
\mathcal{V}\left[f g\right]\leq\max_{x\in[-1,1]}{\lvert f(x)\rvert}\mathcal{V}\left[g\right]+\max_{x\in[-1,1]}{\lvert g(x)\rvert}\mathcal{V}\left[f\right].
\end{align}
\end{lemma}

\begin{proof}
Let $f,g\in\mathcal{BV}\left[-1,1\right]$. Let $\left\{x_1,\dots,x_K\right\}$ be a segmentation of the interval $[-1,1]$, i. e. we have\\
$-1=x_1<x_2<\dots<x_i<x_{i+1}<\dots<x_{K-1}<x_K=1$. Then one has for each $i\in\left\{1,\dots,K-1\right\}$
\begin{align*}
&\left\lvert f\left(x_{i+1}\right) g\left(x_{i+1}\right)-f\left(x_i\right) g\left(x_i\right)\right\rvert\\
\leq&\left\lvert f\left(x_{i+1}\right) g\left(x_{i+1}\right)-f\left(x_{i+1}\right) g\left(x_i\right)\right\rvert+\left\lvert f\left(x_{i+1}\right) g\left(x_i\right)-f\left(x_i\right) g\left(x_i\right)\right\rvert\\
\leq&\max_{x\in[-1,1]}{\lvert f(x)\rvert}\left\lvert g\left(x_{i+1}\right)-g\left(x_i\right)\right\rvert+\max_{x\in[-1,1]}{\lvert g(x)\rvert}\left\lvert f\left(x_{i+1}\right)-f\left(x_i\right)\right\rvert.
\end{align*}
The summation over $i=1,\dots,K-1$ on the left-hand side and on the right-hand side yield to
\begin{align*}
&\sum_{i=1}^{K-1}{\left\lvert f\left(x_{i+1}\right) g\left(x_{i+1}\right)-f\left(x_i\right) g\left(x_i\right)\right\rvert}\\
\leq&\max_{x\in[-1,1]}{\lvert f(x)\rvert}\sum_{i=1}^{K-1}{\left\lvert g\left(x_{i+1}\right)-g\left(x_i\right)\right\rvert}+\max_{x\in[-1,1]}{\lvert g(x)\rvert}\sum_{i=1}^{K-1}{\left\lvert f\left(x_{i+1}\right)-f\left(x_i\right)\right\rvert}\\
\leq&\max_{x\in[-1,1]}{\lvert f(x)\rvert}\mathcal{V}\left[g\right]+\max_{x\in[-1,1]}{\lvert g(x)\rvert}\mathcal{V}\left[f\right].
\end{align*}
Hence the claim follows, since the segmentation is arbitrarily.
\end{proof}

\begin{lemma}\label{trapezformel}
For each $f\in\mathcal{C}\left[-1,1\right]\cap\mathcal{BV}\left[-1,1\right]$ and each $N\in\N$ holds
\begin{align}\label{trapezformela}
\left\lvert\int_{-1}^1{f(x)\mathrm{d}x}-\frac{2}{N}\sum_{i=1}^{N-1}{f\left(\frac{2i}{N}-1\right)-\frac{1}{N}f(-1)-\frac{1}{N}f(1)}\right\rvert\leq\frac{1}{N}\mathcal{V}\left[f\right]
\end{align}
\end{lemma}

\begin{proof}
This result follows directly from \cite[p. 218, (7.14)]{8}.
\end{proof}

We now consider the case $\alpha=\beta=0$ and for the sake of brevity we introduce $P_n\equiv P_n^{0,0}$ and $Q_n(x)\equiv Q_n\left(x,0,0,N\right)$. Then are the weight-functions identical $1$, i. e. $\varrho\equiv 1$ and $\omega\equiv 1$ and we write briefly $(.,.)$ and $\left<.,.\right>$ for the inner products $(.,.)_\varrho$ and $\left<.,.\right>_\omega$. Moreover in the following we write briefly $\left<f,g\right>$ for inner products $\left<f\left(\frac{2}{N}(.)-1\right),g\left(\frac{2}{N}(.)-1\right)\right>$ between two functions $f\colon[-1,1]\rightarrow\R$ and $g\colon[-1,1]\rightarrow\R$.\\
The next Lemmatas \ref{s7} - \ref{s9} are preliminary calculations for the subsequent Theorem \ref{s10}. This Theorem is a special case of our main result Theorem \ref{s16} with the parameter $\alpha=0$.\\
Let in the following
\begin{align}
A_{n,N}:=\frac{(-1)^nN!N!2}{\Gamma(n+2+N)(N-n)!}.
\end{align}

\begin{lemma}\label{s7}
Let for $N\in\N$
\begin{align}
n(N):=\frac{1}{2}+\frac{1}{2}\sqrt{(2N+1)}.
\end{align}
For each $u\in\mathcal{BV}\left[-1,1\right]$ one has
\begin{align}\label{s7a}
\left\lvert\frac{2}{N}\left<u,P_n\right>Q_n\left(\frac{N}{2}(1+x)\right)-\left<u,Q_n\right>Q_n\left(\frac{N}{2}(1+x)\right)A_{n,N}\right\rvert=\mathcal{O}\left(\frac{n^2}{N}\right),
\end{align}
with $n\leq n(N)$, for each $x\in[-1,1]$.
\end{lemma}

\begin{proof}
Let $u\in\mathcal{BV}\left[-1,1\right]$ and let $x\in[-1,1]$. Then we have for any $n\leq n(N)$
\begin{align*}
&\left\lvert\frac{2}{N}\left<u,P_n\right>Q_n\left(\frac{N}{2}(1+x)\right)-\left<u,Q_n\right>Q_n\left(\frac{N}{2}(1+x)\right)A_{n,N}\right\rvert\\
\leq&\left\lvert Q_n\left(\frac{N}{2}(1+x)\right)\right\rvert\left\lvert\frac{2}{N}\left<u,P_n\right>-\left<u,Q_n\right>\frac{(-1)^nN!N!2}{\Gamma(n+2+N)(N-n)!}\right\rvert\\
\overset{(\ref{s1a})}{\leq}&\left\lvert\frac{2}{N}\left<u,P_n\right>-\left<u,Q_n\right>\frac{(-1)^nN!N!2}{\Gamma(n+2+N)(N-n)!}\right\rvert\\
\leq&\frac{2}{N}\left\lvert\left<u,P_n\right>-\left<u,(-1)^nQ_n\right>\frac{N!N!N}{\Gamma(n+2+N)(N-n)!}\right\rvert\\
\leq&\frac{2}{N}\left\lvert\left<u,P_n-(-1)^nQ_n\right>-\left<u,(-1)^nQ_n\right>\left(\frac{N!N!N}{\Gamma(n+2+N)(N-n)!}-1\right)\right\rvert\\
\leq&\max_{x\in[-1,1]}{\lvert u(x)\rvert}\frac{2}{N}\left[\left<1,\left\lvert P_n-(-1)^nQ_n\right\rvert\right>+\left<1,\left\lvert(-1)^nQ_n\right\rvert\right>\left\lvert(-1)^n\frac{N}{2}A_{n,N}-1\right\rvert\right]\\
\overset{(\ref{s6aa})}{\leq}&\max_{x\in[-1,1]}{\lvert u(x)\rvert}\frac{2}{N}\left[\left<1,\left\lvert P_n-(-1)^nQ_n\right\rvert\right>+\left<1,1\right>\mathcal{O}\left(\frac{n^2}{N}\right)\right]\\
\overset{(\ref{s3a})}{\leq}&\max_{x\in[-1,1]}{\lvert u(x)\rvert}\frac{2}{N}\left[\mathcal{O}\left(\frac{n^2}{N}\right)\left<1,1\right>+\left<1,1\right>\mathcal{O}\left(\frac{n^2}{N}\right)\right]\\
\leq&\max_{x\in[-1,1]}{\lvert u(x)\rvert}\frac{2}{N}\left[\mathcal{O}\left(\frac{n^2}{N}\right)(N+1)+(N+1)\mathcal{O}\left(\frac{n^2}{N}\right)\right]\\
\leq&\mathcal{O}\left(\frac{n^2}{N}\right).
\end{align*}
\end{proof}

\begin{lemma}\label{s8}
Let for $N\in\N$
\begin{align}
n(N):=\frac{1}{2}+\frac{1}{2}\sqrt{(2N+1)}.
\end{align}
For each $u\in\mathcal{C}\left[-1,1\right]\cap\mathcal{BV}\left[-1,1\right]$ one has
\begin{align}\label{s8a}
\left\lvert\left(u,P_n\right)P_n(x)-\frac{2}{N}\left<u,P_n\right>(-1)^nQ_n\left(\frac{N}{2}(1+x)\right)\right\rvert=\mathcal{O}\left(\frac{n^2}{N}\right),
\end{align}
with $n\leq n(N)$, for each $x\in[-1,1]$.
\end{lemma}

\begin{proof}
Let $u\in\mathcal{C}\left[-1,1\right]\cap\mathcal{BV}\left[-1,1\right]$ and let $x\in[-1,1]$. Then we have for any $n\leq n(N)$
\begin{align*}
&\left\lvert\left(u,P_n\right)P_n(x)-\frac{2}{N}\left<u,P_n\right>(-1)^nQ_n\left(\frac{N}{2}(1+x)\right)\right\rvert\\
\leq&\left\lvert\left(u,P_n\right)P_n(x)-\frac{2}{N}\left<u,P_n\right>P_n(x)\right\rvert\\
+&\left\lvert\frac{2}{N}\left<u,P_n\right>P_n(x)-\frac{2}{N}\left<u,P_n\right>(-1)^nQ_n\left(\frac{N}{2}(1+x)\right)\right\rvert\\
\overset{(\ref{maxjacobi1a})}{\leq}&\left\lvert\left(u,P_n\right)-\frac{2}{N}\left<u,P_n\right>\right\rvert+\frac{2}{N}\left\lvert\left<u,P_n\right>\right\rvert\left\lvert P_n(x)-(-1)^nQ_n\left(\frac{N}{2}(1+x)\right)\right\rvert\\
\overset{(\ref{s3a})}{\leq}&\left\lvert\left(u,P_n\right)-\frac{2}{N}\sum_{i=0}^N{u\left(\frac{2i}{N}-1\right)P_n\left(\frac{2i}{N}-1\right)}\right\rvert+\frac{2(N+1)}{N}\max_{x\in[-1,1]}{\lvert u(x)\rvert}\mathcal{O}\left(\frac{n^2}{N}\right)\\
\leq&\left\lvert\left(u,P_n\right)-\frac{2}{N}\sum_{i=1}^{N-1}{u\left(\frac{2i}{N}-1\right)P_n\left(\frac{2i}{N}-1\right)-\frac{1}{N}u(-1)P_n(-1)-\frac{1}{N}u(1)P_n(1)}\right\rvert\\
+&\left\lvert\frac{1}{N}u(-1)P_n(-1)\right\rvert+\left\lvert\frac{1}{N}u(1)P_n(1)\right\rvert+\mathcal{O}\left(\frac{n^2}{N}\right).
\end{align*}
With Lemma \ref{trapezformel} follows
\begin{align*}
&\left\lvert\left(u,P_n\right)P_n(x)-\frac{2}{N}\left<u,P_n\right>(-1)^nQ_n\left(\frac{N}{2}(1+x)\right)\right\rvert\\
\overset{(\ref{trapezformela})}{\leq}&\frac{1}{N}\mathcal{V}\left[u P_n\right]+\mathcal{O}\left(\frac{1}{N}\right)+\mathcal{O}\left(\frac{1}{N}\right)+\mathcal{O}\left(\frac{n^2}{N}\right)\\
\overset{(\ref{bvproda})}{\leq}&\frac{1}{N}\left[\max_{x\in[-1,1]}{\lvert u(x)\rvert}\mathcal{V}\left[P_n\right]+\max_{x\in[-1,1]}{\lvert P_n(x)\rvert}\mathcal{V}\left[u\right]\right]+\mathcal{O}\left(\frac{n^2}{N}\right)\\
\overset{(\ref{totvarjacobi1a})}{\leq}&\frac{1}{N}\left[\max_{x\in[-1,1]}{\lvert u(x)\rvert}2n+\mathcal{V}\left[u\right]\right]+\mathcal{O}\left(\frac{n^2}{N}\right)\\
\leq&\mathcal{O}\left(\frac{n}{N}\right)+\mathcal{O}\left(\frac{n^2}{N}\right)\\
\leq&\mathcal{O}\left(\frac{n^2}{N}\right).
\end{align*}
\end{proof}

\begin{lemma}\label{s9}
Let for $N\in\N$
\begin{align}
n(N):=\frac{1}{2}+\frac{1}{2}\sqrt{(2N+1)}.
\end{align}
For each $u\in\mathcal{C}\left[-1,1\right]\cap\mathcal{BV}\left[-1,1\right]$ one has
\begin{align}\label{s9a}
\left\lvert\frac{\left(u,P_n\right)}{\left(P_n,P_n\right)}P_n(x)-\frac{\left<u,Q_n\right>}{\left<Q_n,Q_n\right>}Q_n\left(\frac{N}{2}(1+x)\right)\right\rvert=\mathcal{O}\left(\frac{n^3}{N}\right),
\end{align}
with $n\leq n(N)$, for each $x\in[-1,1]$.
\end{lemma}

\begin{proof}
Let $u\in\mathcal{C}\left[-1,1\right]\cap\mathcal{BV}\left[-1,1\right]$ and let $x\in[-1,1]$. Then we have for any $n\leq n(N)$
\begin{align*}
&\left\lvert\frac{\left(u,P_n\right)}{\left(P_n,P_n\right)}P_n(x)-\frac{\left<u,Q_n\right>}{\left<Q_n,Q_n\right>}Q_n\left(\frac{N}{2}(1+x)\right)\right\rvert\\
\overset{(\ref{s4a})}{=}&\frac{1}{\left\lvert\left(P_n,P_n\right)\right\rvert}\left\lvert\left(u,P_n\right)P_n(x)-\left<u,Q_n\right>Q_n\left(\frac{N}{2}(1+x)\right)(-1)^nA_{n,N}\right\rvert\\
\leq&\frac{1}{\left\lvert\left(P_n,P_n\right)\right\rvert}\left\lvert\left(u,P_n\right)P_n(x)-\frac{2}{N}\left<u,P_n\right>(-1)^nQ_n\left(\frac{N}{2}(1+x)\right)\right\rvert\\
+&\frac{1}{\left\lvert\left(P_n,P_n\right)\right\rvert}\left\lvert\frac{2}{N}\left<u,P_n\right>Q_n\left(\frac{N}{2}(1+x)\right)-\left<u,Q_n\right>Q_n\left(\frac{N}{2}(1+x)\right)A_{n,N}\right\rvert\\
\overset{(\ref{s7a})}{\leq}&\frac{1}{\left\lvert\left(P_n,P_n\right)\right\rvert}\left\lvert\left(u,P_n\right)P_n(x)-\frac{2}{N}\left<u,P_n\right>(-1)^nQ_n\left(\frac{N}{2}(1+x)\right)\right\rvert\\
+&\frac{1}{\left\lvert\left(P_n,P_n\right)\right\rvert}\mathcal{O}\left(\frac{n^2}{N}\right)\\
\overset{(\ref{s8a})}{\leq}&\frac{1}{\left\lvert\left(P_n,P_n\right)\right\rvert}\mathcal{O}\left(\frac{n^2}{N}\right)+\frac{1}{\left\lvert\left(P_n,P_n\right)\right\rvert}\mathcal{O}\left(\frac{n^2}{N}\right)\\
\overset{(\ref{orthogjacobi1a})}{\leq}&\frac{2n+1}{2}\mathcal{O}\left(\frac{n^2}{N}\right)\\
\leq&\mathcal{O}\left(\frac{n^3}{N}\right).
\end{align*}
\end{proof}

\begin{theorem}\label{s10}
Let for $N\in\N$
\begin{align}
n(N):=\frac{1}{2}+\frac{1}{2}\sqrt{(2N+1)}.
\end{align}
For each $u\in\mathcal{C}\left[-1,1\right]\cap\mathcal{BV}\left[-1,1\right]$ one has
\begin{align}
\begin{aligned}
&\left\lvert u(x)-\sum_{k=0}^n{\frac{\left<u,Q_k\right>}{\left<Q_k,Q_k\right>}Q_k\left(\frac{N}{2}(1+x)\right)}\right\rvert\\
\leq&\left\lvert u(x)-\sum_{k=0}^n{\frac{\left(u,P_k\right)}{\left(P_k,P_k\right)}P_k(x)}\right\rvert+\mathcal{O}\left(\frac{n^4}{N}\right),
\end{aligned}
\end{align}
with $n\leq n(N)$, for each $x\in[-1,1]$.
\end{theorem}

\begin{proof}
Let $u\in\mathcal{C}\left[-1,1\right]\cap\mathcal{BV}\left[-1,1\right]$ and let $x\in[-1,1]$. Then we have for any $n\leq n(N)$
\begin{align*}
&\left\lvert u(x)-\sum_{k=0}^n{\frac{\left<u,Q_k\right>}{\left<Q_k,Q_k\right>}Q_k\left(\frac{N}{2}(1+x)\right)}\right\rvert\\
\leq&\left\lvert u(x)-\sum_{k=0}^n{\frac{\left(u,P_k\right)}{\left(P_k,P_k\right)}P_k(x)}\right\rvert\\
+&\left\lvert\sum_{k=0}^n{\frac{\left(u,P_k\right)}{\left(P_k,P_k\right)}P_k(x)}-\sum_{k=0}^n{\frac{\left<u,Q_k\right>}{\left<Q_k,Q_k\right>}Q_k\left(\frac{N}{2}(1+x)\right)}\right\rvert\\
\leq&\left\lvert u(x)-\sum_{k=0}^n{\frac{\left(u,P_k\right)}{\left(P_k,P_k\right)}P_k(x)}\right\rvert\\
+&\sum_{k=0}^n{\left\lvert\frac{\left(u,P_k\right)}{\left(P_k,P_k\right)}P_k(x)-\frac{\left<u,Q_k\right>}{\left<Q_k,Q_k\right>}Q_k\left(\frac{N}{2}(1+x)\right)\right\rvert}\\
\overset{(\ref{s9a})}{\leq}&\left\lvert u(x)-\sum_{k=0}^n{\frac{\left(u,P_k\right)}{\left(P_k,P_k\right)}P_k(x)}\right\rvert+\sum_{k=0}^n{\mathcal{O}\left(\frac{k^3}{N}\right)}\\
\leq&\left\lvert u(x)-\sum_{k=0}^n{\frac{\left(u,P_k\right)}{\left(P_k,P_k\right)}P_k(x)}\right\rvert+\mathcal{O}\left(\frac{n^4}{N}\right).
\end{align*}
\end{proof}

In the following we consider the case $\alpha=\beta>0$ and write briefly\\
$P_n\equiv P_n^{\alpha,\alpha}$ and $Q_n(x)\equiv Q_n\left(x,\alpha,\alpha,N\right)$. Also we are using from Remark \ref{rekurhahn2} 
\begin{align}
\tilde{Q}_n(x):=(-1)^n\binom{n+\alpha}{n}Q_n\left(\frac{N}{2}(1+x);\alpha,\alpha,N\right).
\end{align}
The next Lemmatas \ref{s11} - \ref{s14} are preliminary calculations for the subsequent Theorem \ref{s15}. This Theorem is a special case of our main result Theorem \ref{s16} with the parameter $\alpha>0$.

\begin{lemma}\label{s11}
Let $\alpha>0$. For any $n\leq\frac{N}{2}$ holds
\begin{align}\label{s11a}
\frac{N^{2\alpha}N!N!N}{\Gamma(n+2\alpha+2+N)(N-n)!}=1+\mathcal{O}\left(\frac{n^2}{N}\right).
\end{align}
\end{lemma}

\begin{proof}
Let $\alpha>0$. For any $n\leq\frac{N}{2}$ holds
\begin{align*}
\left\lvert 1-\frac{N^{2\alpha}\Gamma(n+2+N)}{\Gamma(n+2\alpha+2+N)}\right\rvert&=1-\frac{N^{2\alpha}\Gamma(n+2+N)}{\Gamma(n+2\alpha+2+N)}\\
&= 1-\frac{N^{2\alpha}}{(n+2+N)^{2\alpha}}\frac{\Gamma(n+2+N)(n+2+N)^{2\alpha}}{\Gamma(n+2\alpha+2+N)}\\
&\overset{(\ref{gamma1a})}{=}1-\left(1+\mathcal{O}\left(\frac{n^2}{N}\right)\right)\left(1+\mathcal{O}\left(\frac{1}{n+N}\right)\right)\\
&=\mathcal{O}\left(\frac{n^2}{N}\right).
\end{align*}
With Lemma \ref{s6} we have for any $n\leq\frac{N}{2}$
\begin{align*}
\frac{N^{2\alpha}N!N!N}{\Gamma(n+2\alpha+2+N)(N-n)!}&=\frac{N^{2\alpha}\Gamma(n+2+N)}{\Gamma(n+2\alpha+2+N)}\frac{N!N!N}{\Gamma(n+2+N)(N-n)!}\\
&=\left(1+\mathcal{O}\left(\frac{n^2}{N}\right)\right)\left(1+\mathcal{O}\left(\frac{n^2}{N}\right)\right)\\
&=1+\mathcal{O}\left(\frac{n^2}{N}\right).
\end{align*}
\end{proof}

Let in the following
\begin{align}
A_{n,N}^\alpha:=\frac{(-1)^nN!N!2^{2\alpha+1}}{\Gamma(n+2\alpha+2+N)(N-n)!}.
\end{align}

\begin{lemma}\label{s12}
Let $\alpha>0$ and let for $N\in\N$
\begin{align}
n(\alpha,N):=\frac{1}{2}-\alpha+\frac{1}{2}\sqrt{(2\alpha+1)(2\alpha+2N+1)}.
\end{align}
For each $u\in\mathcal{BV}\left[-1,1\right]$ holds
\begin{align}\label{s12a}
\begin{aligned}
&\left\lvert Q_n\left(\frac{N}{2}(1+x)\right)\left[\frac{2}{N}\left<u,P_n\right>_\varrho\binom{n+\alpha}{n}-\left<u,Q_n\right>_\omega\frac{\Gamma^2(\alpha+1+n)}{n!n!}A_{n,N}^\alpha\right]\right\rvert\\
=&\mathcal{O}\left(\frac{n^{1+\alpha+\max{\{1,\alpha\}}}}{N}\right),
\end{aligned}
\end{align}
with $n\leq n(\alpha,N)$, for each $x\in[-1,1]$.
\end{lemma}

\begin{proof}
Let $u\in\mathcal{BV}\left[-1,1\right]$ and let $x\in[-1,1]$. Then we have for any $n\leq n(\alpha,N)$
\begin{align*}
&\left\lvert Q_n\left(\frac{N}{2}(1+x)\right)\left[\frac{2}{N}\left<u,P_n\right>_\varrho\binom{n+\alpha}{n}-\left<u,Q_n\right>_\omega\frac{\Gamma^2(\alpha+1+n)}{n!n!}A_{n,N}^\alpha\right]\right\rvert\\
\leq&\left\lvert Q_n\left(\frac{N}{2}(1+x)\right)\right\rvert\left\lvert\frac{2}{N}\left<u,P_n\right>_\varrho\binom{n+\alpha}{n}-\left<u,Q_n\right>_\omega\frac{\Gamma^2(\alpha+1+n)}{n!n!}A_{n,N}^\alpha\right\rvert\\
\overset{(\ref{s1a})}{\leq}&\binom{n+\alpha}{n}\left\lvert\frac{2}{N}\left<u,P_n\right>_\varrho-\left<u,Q_n\right>_\omega\frac{\Gamma(\alpha+1+n)\Gamma(\alpha+1)}{n!}A_{n,N}^\alpha\right\rvert\\
\overset{(\ref{s5a})}{\leq}&\binom{n+\alpha}{n}\left\lvert\frac{2}{N}\left<u,P_n\right>_\varrho-\left<u,Q_n\right>_\varrho\frac{N^{2\alpha}\left(1+\mathcal{O}\left(\frac{1}{N}\right)\right)}{2^{2\alpha}\Gamma(\alpha+1)}\frac{\Gamma(\alpha+1+n)}{n!}A_{n,N}^\alpha\right\rvert\\
\leq&\binom{n+\alpha}{n}\left\lvert\frac{2}{N}\left<u,P_n\right>_\varrho-\frac{2}{N}\left<u,\tilde{Q}_n\right>_\varrho\frac{N^{2\alpha}N!N!N\left(1+\mathcal{O}\left(\frac{1}{N}\right)\right)}{\Gamma(n+2\alpha+2+N)(N-n)!}\right\rvert
\end{align*}
With Lemma \ref{s11} follows
\begin{align*}
&\left\lvert Q_n\left(\frac{N}{2}(1+x)\right)\left[\frac{2}{N}\left<u,P_n\right>_\varrho\binom{n+\alpha}{n}-\left<u,Q_n\right>_\omega\frac{\Gamma^2(\alpha+1+n)}{n!n!}A_{n,N}^\alpha\right]\right\rvert\\
\overset{(\ref{s11a})}{\leq}&\binom{n+\alpha}{n}\left\lvert\frac{2}{N}\left<u,P_n\right>_\varrho-\frac{2}{N}\left<u,\tilde{Q}_n\right>_\varrho\left(1+\mathcal{O}\left(\frac{n^2}{N}\right)\right)\left(1+\mathcal{O}\left(\frac{1}{N}\right)\right)\right\rvert\\
\leq&\binom{n+\alpha}{n}\left\lvert\frac{2}{N}\left<u,P_n\right>_\varrho+\frac{2}{N}\left<u,P_n-\tilde{Q}_n-P_n\right>_\varrho\left(1+\mathcal{O}\left(\frac{n^2}{N}\right)\right)\right\rvert\\
\leq&\frac{2}{N}\binom{n+\alpha}{n}\left[\left\lvert\left<u,P_n\right>_\varrho\right\rvert\mathcal{O}\left(\frac{n^2}{N}\right)+\left\lvert\left<u,P_n-\tilde{Q}_n\right>_\varrho\right\rvert\left(1+\mathcal{O}\left(\frac{n^2}{N}\right)\right)\right]
\end{align*}
We obtain with Lemma \ref{s3}
\begin{align*}
&\left\lvert Q_n\left(\frac{N}{2}(1+x)\right)\left[\frac{2}{N}\left<u,P_n\right>_\varrho\binom{n+\alpha}{n}-\left<u,Q_n\right>_\omega\frac{\Gamma^2(\alpha+1+n)}{n!n!}A_{n,N}^\alpha\right]\right\rvert\\
\overset{(\ref{s3a})}{\overset{(\ref{s3b})}{\leq}}&\frac{2}{N}\binom{n+\alpha}{n}\left\lvert\left<u,P_n\right>_\varrho\right\rvert\mathcal{O}\left(\frac{n^2}{N}\right)\\
+&\frac{2}{N}\binom{n+\alpha}{n}\max_{x\in[-1,1]}{\lvert u(x)\varrho(x)\rvert}\mathcal{O}\left(\frac{n^{1+\max{\{1,\alpha\}}}}{N}\right)\left<1,1\right>_1\left(1+\mathcal{O}\left(\frac{n^2}{N}\right)\right)\\
\leq&\frac{2}{N}\binom{n+\alpha}{n}\left\lvert\left<u,P_n\right>_\varrho\right\rvert\mathcal{O}\left(\frac{n^2}{N}\right)+\mathcal{O}\left(\frac{n^{1+\alpha+\max{\{1,\alpha\}}}}{N}\right)\\
\leq&\frac{2}{N}\binom{n+\alpha}{n}\max_{x\in[-1,1]}{\lvert u(x)\rvert}\max_{x\in[-1,1]}{\lvert P_n(x)\varrho(x)\rvert}\left<1,1\right>_1\mathcal{O}\left(\frac{n^2}{N}\right)\\
+&\mathcal{O}\left(\frac{n^{1+\alpha+\max{\{1,\alpha\}}}}{N}\right)\\
\leq&\max_{x\in[-1,1]}{\lvert P_n(x)\varrho(x)\rvert}\mathcal{O}\left(\frac{n^{2+\alpha}}{N}\right)+\mathcal{O}\left(\frac{n^{1+\alpha+\max{\{1,\alpha\}}}}{N}\right)\\
\overset{(\ref{totvarjacobi4a})}{\overset{(\ref{totvarjacobi4b})}{\leq}}&\mathcal{O}\left(\frac{n^{1+\alpha+\max{\{1,\alpha\}}}}{N}\right).
\end{align*}
\end{proof}

\begin{lemma}\label{s13}
Let $\alpha>-\frac{1}{2}$ and let for $N\in\N$
\begin{align}
n(\alpha,N):=\frac{1}{2}-\alpha+\frac{1}{2}\sqrt{(2\alpha+1)(2\alpha+2N+1)}.
\end{align}
For each $u\in\mathcal{C}\left[-1,1\right]\cap\mathcal{BV}\left[-1,1\right]$ and $\frac{1}{2}\geq\alpha>0$ holds
\begin{align}\label{s13a}
\left\lvert\left(u,P_n\right)_\varrho P_n(x)-\frac{2}{N}\left<u,P_n\right>_\varrho\tilde{Q}_n(x)\right\rvert=\mathcal{O}\left(\frac{n^{2-\alpha}}{N}\right),
\end{align}
with $n\leq n(\alpha,N)$, for each $x\in[-1,1]$.\\
For each $u\in\mathcal{C}\left[-1,1\right]\cap\mathcal{BV}\left[-1,1\right]$ and $\alpha\geq\frac{1}{2}$ holds
\begin{align}\label{s13b}
\left\lvert\left(u,P_n\right)_\varrho P_n(x)-\frac{2}{N}\left<u,P_n\right>_\varrho\tilde{Q}_n(x)\right\rvert=\mathcal{O}\left(\frac{n^{\frac{1}{2}+\max{\{1,\alpha\}}}}{N}\right),
\end{align}
with $n\leq n(\alpha,N)$, for each $x\in[-1,1]$.
\end{lemma}

\begin{proof}
Let $\alpha>0$, let $u\in\mathcal{C}\left[-1,1\right]\cap\mathcal{BV}\left[-1,1\right]$ and let $x\in[-1,1]$. Then we have for any $n\leq n(\alpha,N)$
\begin{align*}
&\left\lvert\left(u,P_n\right)_\varrho P_n(x)-\frac{2}{N}\left<u,P_n\right>_\varrho\tilde{Q}_n(x)\right\rvert\\
\leq&\left\lvert\left(u,P_n\right)_\varrho P_n(x)-\frac{2}{N}\left<u,P_n\right>_\varrho P_n(x)\right\rvert\\
+&\left\lvert\frac{2}{N}\left<u,P_n\right>_\varrho P_n(x)-\frac{2}{N}\left<u,P_n\right>_\varrho\tilde{Q}_n(x)\right\rvert\\
\overset{(\ref{maxjacobi1a})}{\leq}&\binom{n+\alpha}{n}\left\lvert\left(u,P_n\right)_\varrho-\frac{2}{N}\left<u,P_n\right>_\varrho\right\rvert+\frac{2}{N}\left\lvert\left<u,P_n\right>_\varrho\right\rvert\left\lvert P_n(x)-\tilde{Q}_n(x)\right\rvert\\
\overset{(\ref{trapezformela})}{\leq}&\binom{n+\alpha}{n}\frac{1}{N}\mathcal{V}\left[u P_n \varrho\right]+\frac{2}{N}\left\lvert\left<u,P_n\right>_\varrho\right\rvert\left\lvert P_n(x)-\tilde{Q}_n(x)\right\rvert
\end{align*}
With Lemma \ref{s3} follows
\begin{align*}
&\left\lvert\left(u,P_n\right)_\varrho P_n(x)-\frac{2}{N}\left<u,P_n\right>_\varrho\tilde{Q}_n(x)\right\rvert\\
\overset{(\ref{s3a})}{\overset{(\ref{s3b})}{\leq}}&\binom{n+\alpha}{n}\frac{1}{N}\mathcal{V}\left[u P_n \varrho\right]\\
+&\frac{2(N+1)}{N}\max_{x\in[-1,1]}{\lvert u(x)\rvert}\max_{x\in[-1,1]}{\lvert P_n(x) \varrho(x)\rvert}\mathcal{O}\left(\frac{n^{1+\max{\{1,\alpha\}}}}{N}\right)\\
\overset{(\ref{bvproda})}{\leq}&\binom{n+\alpha}{n}\frac{1}{N}\left[\max_{x\in[-1,1]}{\lvert u(x)\rvert}\mathcal{V}\left[P_n \varrho\right]+\max_{x\in[-1,1]}{\lvert P_n(x) \varrho(x)\rvert}\mathcal{V}\left[u\right]\right]\\
+&\max_{x\in[-1,1]}{\lvert P_n(x) \varrho(x)\rvert}\mathcal{O}\left(\frac{n^{1+\max{\{1,\alpha\}}}}{N}\right)\\
\leq&\left[\mathcal{V}\left[P_n \varrho\right]+\max_{x\in[-1,1]}{\lvert P_n(x) \varrho(x)\rvert}\right]\mathcal{O}\left(\frac{n^{\alpha}}{N}\right)\\
+&\max_{x\in[-1,1]}{\lvert P_n(x) \varrho(x)\rvert}\mathcal{O}\left(\frac{n^{1+\max{\{1,\alpha\}}}}{N}\right)\\
\leq&\mathcal{V}\left[P_n \varrho\right]\mathcal{O}\left(\frac{n^{\alpha}}{N}\right)+\max_{x\in[-1,1]}{\lvert P_n(x) \varrho(x)\rvert}\mathcal{O}\left(\frac{n^{1+\max{\{1,\alpha\}}}}{N}\right).
\end{align*}
Then we have especially for $\frac{1}{2}\geq\alpha>0$
\begin{align*}
&\left\lvert\left(u,P_n\right)_\varrho P_n(x)-\frac{2}{N}\left<u,P_n\right>_\varrho\tilde{Q}_n(x)\right\rvert\\
\leq&\mathcal{V}\left[P_n \varrho\right]\mathcal{O}\left(\frac{n^{\alpha}}{N}\right)+\max_{x\in[-1,1]}{\lvert P_n(x) \varrho(x)\rvert}\mathcal{O}\left(\frac{n^{1+\max{\{1,\alpha\}}}}{N}\right)\\
\overset{(\ref{totvarjacobi4b})}{\overset{(\ref{totvarjacobi5b})}{\leq}}&\mathcal{O}\left(n^{1-\alpha}\right)\mathcal{O}\left(\frac{n^{\alpha}}{N}\right)+\mathcal{O}\left(n^{-\alpha}\right)\mathcal{O}\left(\frac{n^{1+\max{\{1,\alpha\}}}}{N}\right)\\
\leq&\mathcal{O}\left(\frac{n^{2-\alpha}}{N}\right).
\end{align*}
And we have especially for $\alpha\geq\frac{1}{2}$
\begin{align*}
&\left\lvert\left(u,P_n\right)_\varrho P_n(x)-\frac{2}{N}\left<u,P_n\right>_\varrho\tilde{Q}_n(x)\right\rvert\\
\leq&\mathcal{V}\left[P_n \varrho\right]\mathcal{O}\left(\frac{n^{\alpha}}{N}\right)+\max_{x\in[-1,1]}{\lvert P_n(x) \varrho(x)\rvert}\mathcal{O}\left(\frac{n^{1+\max{\{1,\alpha\}}}}{N}\right)\\
\overset{(\ref{totvarjacobi4a})}{\overset{(\ref{totvarjacobi5a})}{\leq}}&\mathcal{O}\left(n^{\frac{1}{2}}\right)\mathcal{O}\left(\frac{n^{\alpha}}{N}\right)+\mathcal{O}\left(n^{-\frac{1}{2}}\right)\mathcal{O}\left(\frac{n^{1+\max{\{1,\alpha\}}}}{N}\right)\\
\leq&\mathcal{O}\left(\frac{n^{\frac{1}{2}+\max{\{1,\alpha\}}}}{N}\right).
\end{align*}
\end{proof}

\begin{lemma}\label{s14}
Let $\alpha>0$ and let for $N\in\N$
\begin{align}
n(\alpha,N):=\frac{1}{2}-\alpha+\frac{1}{2}\sqrt{(2\alpha+1)(2\alpha+2N+1)}.
\end{align}
For each $u\in\mathcal{C}\left[-1,1\right]\cap\mathcal{BV}\left[-1,1\right]$ holds
\begin{align}\label{s14a}
\left\lvert\frac{\left(u,P_n\right)_\varrho}{\left(P_n,P_n\right)_\varrho}P_n(x)-\frac{\left<u,Q_n\right>_\omega}{\left<Q_n,Q_n\right>_\omega}Q_n\left(\frac{N}{2}(1+x)\right)\right\rvert=\mathcal{O}\left(\frac{n^{2+\alpha+\max{\{1,\alpha\}}}}{N}\right),
\end{align}
with $n\leq n(\alpha,N)$, for each $x\in[-1,1]$.
\end{lemma}

\begin{proof}
Let $u\in\mathcal{C}\left[-1,1\right]\cap\mathcal{BV}\left[-1,1\right]$ and let $x\in[-1,1]$. Then we have for any $n\leq n(\alpha,N)$
\begin{align*}
&\left\lvert\frac{\left(u,P_n\right)_\varrho}{\left(P_n,P_n\right)_\varrho}P_n(x)-\frac{\left<u,Q_n\right>_\omega}{\left<Q_n,Q_n\right>_\omega}Q_n\left(\frac{N}{2}(1+x)\right)\right\rvert\\
\overset{(\ref{s4a})}{=}&\frac{1}{\left\lvert\left(P_n,P_n\right)_\varrho\right\rvert}\left\lvert\left(u,P_n\right)_\varrho P_n(x)-\left<u,Q_n\right>_\omega\tilde{Q}_n(x)\frac{\Gamma(\alpha+1+n)\Gamma(\alpha+1)}{n!}A_{n,N}^\alpha\right\rvert\\
\leq&\frac{1}{\left\lvert\left(P_n,P_n\right)_\varrho\right\rvert}\left\lvert\left(u,P_n\right)_\varrho P_n(x)-\frac{2}{N}\left<u,P_n\right>_\varrho\tilde{Q}_n(x)\right\rvert\\
+&\frac{1}{\left\lvert\left(P_n,P_n\right)_\varrho\right\rvert}\left\lvert\tilde{Q}_n(x)\left[\frac{2}{N}\left<u,P_n\right>_\varrho-\left<u,Q_n\right>_\omega\frac{\Gamma(\alpha+1+n)\Gamma(\alpha+1)}{n!}A_{n,N}^\alpha\right]\right\rvert\\
\overset{(\ref{s12a})}{\leq}&\frac{1}{\left\lvert\left(P_n,P_n\right)_\varrho\right\rvert}\left\lvert\left(u,P_n\right)_\varrho P_n(x)-\frac{2}{N}\left<u,P_n\right>_\varrho\tilde{Q}_n(x)\right\rvert\\
+&\frac{1}{\left\lvert\left(P_n,P_n\right)_\varrho\right\rvert}\mathcal{O}\left(\frac{n^{1+\alpha+\max{\{1,\alpha\}}}}{N}\right)\\
\overset{(\ref{s13a})}{\overset{(\ref{s13b})}{\leq}}&\frac{1}{\left\lvert\left(P_n,P_n\right)_\varrho\right\rvert}\mathcal{O}\left(\frac{n^{1+\max{\{1,\alpha\}}}}{N}\right)+\frac{1}{\left\lvert\left(P_n,P_n\right)_\varrho\right\rvert}\mathcal{O}\left(\frac{n^{1+\alpha+\max{\{1,\alpha\}}}}{N}\right)\\
\overset{(\ref{orthogjacobi1a})}{\leq}&\frac{(2n+2\alpha+1)n!\Gamma(n+2\alpha+1)}{2^{2\alpha+1}\Gamma(n+\alpha+1)\Gamma(n+\alpha+1)}\mathcal{O}\left(\frac{n^{1+\alpha+\max{\{1,\alpha\}}}}{N}\right)\\
\overset{(\ref{gamma1a})}{\leq}&\mathcal{O}\left(\frac{n^{2+\alpha+\max{\{1,\alpha\}}}}{N}\right).
\end{align*}
\end{proof}

\begin{theorem}\label{s15}
Let $\alpha>0$ and let for $N\in\N$
\begin{align}
n(\alpha,N):=\frac{1}{2}-\alpha+\frac{1}{2}\sqrt{(2\alpha+1)(2\alpha+2N+1)}.
\end{align}
For each $u\in\mathcal{C}\left[-1,1\right]\cap\mathcal{BV}\left[-1,1\right]$ holds
\begin{align}
\begin{aligned}
&\left\lvert u(x)-\sum_{k=0}^n{\frac{\left<u,Q_k\right>_\omega}{\left<Q_k,Q_k\right>_\omega}Q_k\left(\frac{N}{2}(1+x)\right)}\right\rvert\\
\leq&\left\lvert u(x)-\sum_{k=0}^n{\frac{\left(u,P_k\right)_\varrho}{\left(P_k,P_k\right)_\varrho}P_k(x)}\right\rvert+\mathcal{O}\left(\frac{n^{3+\alpha+\max{\{1,\alpha\}}}}{N}\right),
\end{aligned}
\end{align}
with $n\leq n(\alpha,N)$, for each $x\in[-1,1]$.
\end{theorem}

\begin{proof}
Let $u\in\mathcal{C}\left[-1,1\right]\cap\mathcal{BV}\left[-1,1\right]$ and let $x\in[-1,1]$. Then we have for any $n\leq n(\alpha,N)$
\begin{align*}
&\left\lvert u(x)-\sum_{k=0}^n{\frac{\left<u,Q_k\right>_\omega}{\left<Q_k,Q_k\right>_\omega}Q_k\left(\frac{N}{2}(1+x)\right)}\right\rvert\\
\leq&\left\lvert u(x)-\sum_{k=0}^n{\frac{\left(u,P_k\right)_\varrho}{\left(P_k,P_k\right)_\varrho}P_k(x)}\right\rvert\\
+&\left\lvert\sum_{k=0}^n{\frac{\left(u,P_k\right)_\varrho}{\left(P_k,P_k\right)_\varrho}P_k(x)}-\sum_{k=0}^n{\frac{\left<u,Q_k\right>_\omega}{\left<Q_k,Q_k\right>_\omega}Q_k\left(\frac{N}{2}(1+x)\right)}\right\rvert\\
\leq&\left\lvert u(x)-\sum_{k=0}^n{\frac{\left(u,P_k\right)_\varrho}{\left(P_k,P_k\right)_\varrho}P_k(x)}\right\rvert\\
+&\sum_{k=0}^n{\left\lvert\frac{\left(u,P_k\right)_\varrho}{\left(P_k,P_k\right)_\varrho}P_k(x)-\frac{\left<u,Q_k\right>_\omega}{\left<Q_k,Q_k\right>_\omega}Q_k\left(\frac{N}{2}(1+x)\right)\right\rvert}\\
\overset{(\ref{s14a})}{\leq}&\left\lvert u(x)-\sum_{k=0}^n{\frac{\left(u,P_k\right)_\varrho}{\left(P_k,P_k\right)_\varrho}P_k(x)}\right\rvert+\sum_{k=0}^n{\mathcal{O}\left(\frac{k^{2+\alpha+\max{\{1,\alpha\}}}}{N}\right)}\\
\leq&\left\lvert u(x)-\sum_{k=0}^n{\frac{\left(u,P_k\right)_\varrho}{\left(P_k,P_k\right)_\varrho}P_k(x)}\right\rvert+\mathcal{O}\left(\frac{n^{3+\alpha+\max{\{1,\alpha\}}}}{N}\right).
\end{align*}
\end{proof}

\newpage
\addcontentsline{toc}{section}{Literatur}

\bibliography{literatur}

\end{document}